\documentclass[12pt, english]{article}

\usepackage[english]{babel}
\usepackage[applemac]{inputenc}
\usepackage[T1]{fontenc} 

\usepackage{pgf}
\usepackage{tikz}
\usetikzlibrary{arrows}
\usetikzlibrary{decorations.pathmorphing}
\usetikzlibrary{backgrounds}
\usetikzlibrary{folding}
\usepackage{afterpage}
\usetikzlibrary{fit}

\usepackage{amsmath}
\usepackage{amssymb}
\usepackage{amsthm}
\usepackage{url}

\theoremstyle{plain}
\newtheorem{Thm}{Theorem}
\newtheorem{Prop}[Thm]{Proposition}
\newtheorem{Lemma}[Thm]{Lemma}

\theoremstyle{definition}

\begin{document}

\title{Padlock Solitaire: A martingale trick for combinatorial enumeration} 
\author{Johan Wästlund}
\date{\today}   

\maketitle

\begin{abstract}
We introduce a one-person game that we call Padlock Solitaire which resembles the well-known clock solitaire card game. Analyzing variants of this game we obtain simple proofs of some classical results of combinatorics including ballot theorems and the enumeration of spanning trees in various graphs and hypergraphs. 
\end{abstract}

\section{Boxes and padlocks}
Suppose we have $n$ boxes labeled $1,\dots, n$, each with a padlock with a unique key. We keep the key to the first box, but put the remaining $n-1$ keys randomly into the $n$ boxes and lock them. We assume to begin with that the keys are distributed uniformly and independently. 

Using key 1 we can open box 1, and if that box contains one or more keys, we can keep opening boxes. We call this \emph{padlock solitaire}, and the condition for success, or ``winning'', is that we finally recover all the keys, equivalently unlock all boxes.

The arrangement of keys into boxes can be represented by a directed graph on vertices $1,\dots, n$ with an edge $i \to j$ if key $j$ is in box $i$ (so that opening box $i$ leads to opening box $j$). 

\begin{Prop} \label{P:tree}
We recover all the keys if and only if the distribution of keys into boxes describes a tree rooted at box 1 and directed away from the root.
\end{Prop}

\begin{proof}
The boxes that we can open are precisely those that have a path to them from box 1. If there is such a path to every vertex, then since there are only $n-1$ edges, the graph must be a rooted tree.
\end{proof}

\section{The Cayley formula}
\begin{Prop} \label{P:cayley}
There are $n^{n-2}$ different trees on $n$ labeled vertices.
\end{Prop}

This was established already by James Joseph Sylvester in 1857 \cite{Sylvester} and mentioned by Carl Wilhelm Borchardt \cite{Borchardt}, but is named after Arthur Cayley who devoted the paper \cite{Cayley} to it in 1889. Many beautiful proofs are known, see for instance \cite{AignerZiegler, Joyal, Pitman, Prufer, Renyi, Riordan}. The history is interesting, as the formula can also be derived from the matrix tree theorem whose roots can be traced back to Gustav Kirchhoff in 1847 \cite{KirbyEtAl, Kirchhoff}. 
 
Counting trees, we can choose arbitrarily a vertex as the root, thereby defining for every edge a direction away from the root. The Cayley formula therefore equivalently counts trees with a specified root and direction, for instance the winning starting configurations of padlock solitaire. 

There are obviously $n^{n-1}$ different ways of distributing the $n-1$ keys into the $n$ boxes. An equivalent form of the Cayley formula is therefore that the probability of winning padlock solitaire is $1/n$. We give an essentially calculation-free proof of this fact.

\begin{Thm} \label{T:simple}
If we put the keys $2,\dots, n$ uniformly and independently into the $n$ boxes, the probability of winning is $1/n$.
\end{Thm}

\begin{proof}
The synopsis is that throughout the game, the average number of hidden keys per unopened box is a martingale, and therefore equal to our probability of not winning. But let's fill in some details:

At every moment we consider the ratio of the number of keys still locked in some box to the number of boxes that are locked. 
This ratio becomes 1 if we get stuck with all the remaining keys locked away, and 0 if we recover the last key. If we actually open the last box, the ratio will become $0/0$, but we can regard the game as finished once we find the last key.  

To establish the martingale property, notice that the unseen keys are equally likely to be in any of the remaining boxes, so that the expected change when we open one of them is zero.

Being a martingale that stops at 0 or 1 depending on the outcome, the hidden keys to unopened boxes ratio must be equal to our probability of not winning. It might be more convenient to think instead about the ratio of \emph{available but unused} keys to locked boxes, since that must consequently be our probability of winning. Starting with one out of $n$ keys, our winning probability is therefore $1/n$.
\end{proof}

\section{Rooted forests}

A generalization of Proposition~\ref{P:cayley} \cite{Cayley, Renyi, Riordan, TakacsTree} states that if $k$ of the $n$ labeled vertices are designated as roots, then there are  
\[k\cdot n^{n-k-1}\] 
forests of $n-k$ edges that connect every other vertex to one of the $k$ roots. 

This too follows from the martingale property of the keys-to-boxes ratio: If we start the game with keys $1,\dots, k$, and distribute the remaining $n-k$ keys into the $n$ boxes, the probability of winning is $k/n$. This means that out of the $n^{n-k}$ ways of distributing the $n-k$ keys, the proportion $k/n$ gives the required type of forest. 

\section{Parking functions}
Again let's start with just one key and $n$ boxes. Suppose that even if we get stuck, we open the remaining boxes (with a master key). Then there is an equivalent criterion for winning that we call the \emph{solvency condition}: 

\begin{Prop} \label{P:solvency}
We can unlock all boxes if and only if for every $k<n$, the first $k$ boxes that we open together contain at least $k$ keys.
\end{Prop}

\begin{proof}
If the solvency condition holds, we remain ``solvent'' throughout the game in the sense of always having yet another key, while if it fails for some $k$, we can't open more than at most those $k$ boxes.
\end{proof}

We already know that the probability of winning padlock solitaire is $1/n$. By symmetry, the probability that the solvency condition holds must be the same, $1/n$, if we just open the boxes from left to right using the master key. 

This fact lets us enumerate so-called \emph{parking functions}. These were introduced in \cite{KonheimWeiss}, see also \cite{RiordanParking} and the set of exercises of Section 6.4 of \cite{Knuth}. 
A parking function from $\{1,\dots, n-1\}$ to $\{1,\dots, n-1\}$ is a function where, for each $k\leq n-1$, at least $k$ elements are mapped to $\{1,\dots, k\}$. If \emph{a priori} we allow the value $n$ (even though it can't occur for a parking function), the probability that a uniformly chosen function $\{1,\dots, n-1\} \to \{1,\dots, n\}$ is a parking function is $1/n$, and we conclude that there are $n^{n-2}$ such functions.

\section{Independence is not quite needed}
An interesting aspect of the proof of Theorem~\ref{T:simple} is that we don't quite need to assume that the keys are distributed independently of each other. 

\begin{Thm} \label{T:dependent}
Suppose we keep key 1 and put keys $2,\dots, n$ into boxes $1,\dots, n$ in such a way that whenever we condition on the contents of a set of boxes, each key not found in those boxes is distributed uniformly between the remaining ones. Then the probability of winning is still  $1/n$. 
\end{Thm}
\begin{proof}
At any time, each hidden key is equally likely to be in any of the unopened boxes, and therefore again the average number of keys in those boxes is a martingale.
\end{proof}


A simple example covered by Theorem~\ref{T:dependent} but not by Theorem~\ref{T:simple} is distributing the keys according to a permutation (no two keys in the same box), except that we still keep key 1. We recover all keys if and only if the permutation is cyclic, which happens with probability $1/n$. 

Again we can modify the game by retaining $k$ keys, and conclude that for a uniform random permutation of $\{1,\dots,n\}$, the probability that every cycle contains one of $k$ given numbers is $k/n$. This was the topic of the blog post \cite{PossiblyWrong} (whose author seems to prefer to remain anonymous) that I recently stumbled upon. What led to this note was the observation that the winning probability stays the same even if we allow several keys in the same box.  

Notice though that we must assume something more than just uniform distribution of each key individually: Suppose $n=3$ and we choose uniformly between the three options of putting one of the keys 2 and 3 in its own box and the other in the first box, or putting keys 2 and 3 in each other's boxes (in other words the three odd permutations). Then both keys 2 and 3 are distributed uniformly between the three boxes, but we can never recover both of them.  

\section{Trees in hypergraphs}
The Cayley formula can be generalized to spanning trees in so-called uniform hypergraphs. We show how this works in the case of spanning by triangles, but the result can easily be generalized to hyperedges connecting more than three vertices. 

Suppose there are $2n+1$ labeled vertices and we wish to connect them by $n$ hyperedges, each of which can be thought of as a triangular membrane connecting three of the vertices. If we insert the hyperedges one at a time, each one can decrease the number of components by at most 2, and therefore $n$ hyperedges barely suffice to make the structure connected.  

\begin{Prop}
The number of ways of choosing $n$ triangular hyperedges to connect $2n+1$ labeled vertices is
\begin{equation} \label{Three} 1\cdot 3 \cdot 5 \ \cdots \ (2n-1)\cdot (2n+1)^{n-1}.\end{equation}
\end{Prop}

I believe that this has been known for some time, but I haven't found it stated explicitly other than in \cite{Siva}.
Here we show how to derive \eqref{Three} from padlock solitaire. 

We assume that there are $2n+1$ boxes, one for each vertex. As usual we keep key 1, but before distributing the remaining $2n$ keys into the boxes, we pair them up in one of the $1\cdot 3\cdot 5\ \cdots \ (2n-1)$ ways. Then we distribute the $n$ \emph{pairs} of keys uniformly and independently into the $2n+1$ boxes. Notice that the condition of Theorem~\ref{T:dependent} is satisfied. 

Each way of pairing and distributing the keys can be described by the collection of triples $(i,j,k)$ such that keys $j$ and $k$ are paired up and placed in box $i$. Again it can be verified that we recover all the keys if and only if the resulting structure is connected, and by Theorem~\ref{T:dependent} this happens with probability $1/(2n+1)$.

Therefore out of the $1\cdot 3 \cdot 5 \ \cdots \ (2n-1)\cdot (2n+1)^n$ different ways of pairing up the keys and then distributing the $n$ pairs into the $2n+1$ boxes, a fraction $1/(2n+1)$ will correspond to connected hypergraphs. We conclude that the number of such hypergraphs is given by \eqref{Three}.

\section{Non-uniform key distribution} \label{S:nonuniform}
There is a straightforward generalization of Theorem~\ref{T:dependent} to distributions where not all boxes are equally likely to hold the hidden keys.  

\begin{Thm} \label{T:nonuniform}
Suppose that $p_1+\dots +p_n = 1$ and that each key $2,\dots, n$ is put into box $i$ with probability $p_i$. Suppose moreover that if we condition on the contents of a set of boxes, each key which is not in any of those is distributed between the remaining ones with probabilities proportional to the initial probabilities $p_i$. 
Then the probability of unlocking all boxes is $p_1$.
\end{Thm}

\begin{proof}
The ratio 
\[ \frac{\sum p_i\ \text{(key $i$ locked in)}}{\sum p_j\ \text{(box $j$ locked)}}\] is a martingale, and is therefore equal to the probability of not winning. 
\end{proof}

Simple examples of this type of distribution can be simulated as solitaire card games. Suppose for instance that we shuffle a standard deck of 52 cards and deal 13 piles of three cards each. The piles represent boxes labeled Ace, $2, \dots, 10$, Jack, Queen and King, and the cards of the heart suit are the keys according to the labeling.  

We start with the thirteen remaining cards on our hand. Everything except hearts is thrown into a discard pile, but each card of the heart suit lets us pick up the corresponding pile and obtain three new cards. We win if we recover all hearts and pick up everything. 

If we regard our initial hand of 13 cards as box 1, this is an example of padlock solitaire, and according to Theorem~\ref{T:nonuniform}, our winning probability is $13/52 = 1/4$. 

We can cast the whole proof in terms of card play by noting that since all we do is turn up cards of a shuffled deck, the proportion of hearts among the unseen cards is a martingale. This proportion starts at $1/4$ before we even look at our hand, and ends at 0 if we win, and at $1/3$ (exactly!) if we lose.

As a ``brain teaser'', we suggest the problem of changing the sizes of the thirteen piles (even allowing some piles to be empty) in order to maximize the probability of winning. The answer is that the winning probability depends only on the number of cards in our initial hand, and not on the sizes of the piles. In particular it stays the same if we put all 39 cards not in our hand into the \emph{Ace} pile, and in that case we obviously win if and only if the heart ace is in our hand. 

\section{Clock solitaire}
Our arguments bear strong resemblance to the analysis of the well-known ``clock solitaire'' (also known as ``clock patience'' and under other names like ``travellers''). A fact that has been rediscovered many times is that for this game, the probability of winning is $1/13$. We turn cards over until we have seen all four kings, winning if this happens at the very last card. The order in which we turn the cards over is governed by the cards we see, but the probability of winning is still the same as the probability that the bottom card of a shuffled deck is a king. One way of convincing oneself of the correctness of this conclusion is to note that the proportion of kings among the unseen cards is a martingale.  

A similar argument is explored in the game ``Next Card Red'' in \cite{Winkler}.


\section{Trees with given degree sequence}
Using Theorem~\ref{T:nonuniform} we can count trees on $n$ labeled vertices with prescribed degrees. For the history of this and similar results we refer to \cite{Moon}.

\begin{Prop}
The number of trees on $n$ labeled vertices with prescribed degrees $d_1, \dots, d_n$ is given by the multinomial coefficient
\[ \binom{n-2}{d_1-1,\dots,d_n-1}.\]
\end{Prop} 

\begin{proof}
We play padlock solitaire keeping key 1, and conditioning on exactly $d_1$ keys in box 1 and $d_i-1$ keys in box $i$ for $2\leq i\leq n$. Notice that these numbers must sum to $n-1$ i order for $d_1,\dots, d_n$ to be the degree sequence of a tree, and that the distribution satisfies the condition of Theorem~\ref{T:nonuniform}. 

In order for us to win the game, the location of the keys must describe a tree rooted at box 1. Disregarding the orientation of the edges, there will be, for each box $i$, one edge to every box whose key is in box $i$, and for $i\neq 1$, one edge to the box that contains key $i$. Therefore the winning starting positions are precisely the trees that we wish to count. By Theorem~\ref{T:nonuniform} the probability of winning is $d_1/(n-1)$, and therefore out of the 
\[ \binom{n-1}{d_1, d_2-1,\dots,d_n-1}\] ways of distributing the keys according to the given constraints, the number that describe a tree is
\[ \frac{d_1}{n-1} \cdot \binom{n-1}{d_1, d_2-1,\dots,d_n-1} =  \binom{n-2}{d_1-1,\dots,d_n-1}.\] 
\end{proof}

\section{Parentheses and Catalan numbers}
A string of length $2n$ consisting of $n$ left-parentheses and $n$ right-parentheses is \emph{well-formed} 
 if, reading from left to right, we never see an excess of right-parentheses. The number of well-formed strings of $n$ pairs of parentheses is the $n$:th Catalan number
\[ C_n = \frac{(2n)!}{n!(n+1)!},\]
named after Eugène Catalan who established this result in 1838 \cite{Catalan}. For the history of the Catalan number sequence we refer to \cite{Pak}.

To cast Catalan's result in terms of a solitaire card game as in Section~\ref{S:nonuniform}, suppose we shuffle a deck of $n$ red and $n+1$ black cards and deal into $n+1$ piles, the first one consisting of a single card and the others of two cards each. The first pile is the one to which we have the key, and the $n$ red cards represent the keys to the remaining $n$ piles.

In order for the solvency condition to hold, we must not until we turn over the very last card get a situation where among the cards we have seen there is a black majority. In particular the first card must be red.

According to Theorem~\ref{T:nonuniform}, the probability of winning is $1/(2n+1)$. By symmetry, the probability that the solvency condition holds is the same if we turn the cards over from left to right (starting with the single-card pile). 

We conclude that the probability of getting a well-formed parenthetical expression followed by a final unmatched black card is $1/(2n+1)$, and that therefore the number of well-formed strings of $n$ pairs of parentheses is
\[ \frac1{2n+1}\cdot \binom {2n+1}{n} = \frac{(2n)!}{n!(n+1)!} = C_n.\]

Curiously, we can derive the same result from a different distribution of the keys, again into $n+1$ boxes: We shuffle a deck of $n$ red and $n$ black cards, again letting the red cards represent the $n$ hidden keys. Every time we open a box, we deal cards until we get a black one. The red cards before that are the keys in the box. Such a distribution is in fact invariant under permutations of the boxes, and again we win provided we never see a majority of black cards. The winning probability is now $1/(n+1)$, which gives the alternative expression \[C_n = \frac1{n+1}\cdot \binom{2n}{n}\] for the Catalan numbers, but we omit the details.

 \section{Ballot theorems}
The enumeration of parenthetical expressions belongs to the classical family of \emph{ballot theorems}. Here we give a couple of examples of solitaire card games where the analysis leads to ballot-style results. 

Recall the game of Section~\ref{S:nonuniform} where we deal 13 three-card piles and retain 13 cards. Notice that if we get stuck, the remaining piles will contain exactly the set of hearts (key cards) of the labels of those piles themselves. In particular, exactly one third of the unseen cards will be hearts.

As in the folklore analysis of clock solitaire, we can play a ``lazy'' version, not actually dealing the piles but instead giving ourselves thirteen cards and using the remaining deck as a talon. We play our cards to a discard pile, drawing three new cards from the talon every time we play a heart, and we win if we finish it. Since the solvency condition remains the same in the lazy version, the winning probability is still $13/52 = 1/4$.

Again whenever we lose, exactly one third of the cards remaining in the talon are hearts. There is also a converse: If we peek at the cards from the bottom of the talon three by three, then whenever we find that a third (or more) of the cards from the bottom are hearts, it's clear that we can't win: There won't be enough hearts earlier on to dig that deep into the talon. 

We conclude that if we turn over the cards of a shuffled deck one by one, the probability that the proportion of hearts ever reaches $1/3$ (or more) is exactly $3/4$. This is a special case of a ballot theorem stated by Émile Barbier in 1887 \cite{AddarioBerryReed, Barbier, Renault}.

We  can even establish a ``non-uniform'' ballot theorem. Again we consider a simple example using a standard deck. Suppose we deal ourselves 12 cards and place the remaining 40 as a talon. We play to a discard pile, but this time we pick up new cards according to the ``high-card point'' scale of bridge: A jack (of any suit) gives 1 new card, a queen 2, a king 3 and an ace 4 new cards. Again we win if we finish the talon. Notice that the total number of high-card points is 40, the same as the initial number of cards in the talon.

By the familiar analysis, the winning probability is now $12/52 = 3/13$, and whenever we lose, the remainder of the talon will contain exactly as many high-card points as cards. Consequently if we turn cards over from the bottom, the probability of getting a set of cards with at least as many high-card points as cards is exactly $40/52 = 10/13$.  
 This is a special case of a theorem proved independently by J.~C.~Tanner \cite{Tanner}, Meyer Dwass \cite{Dwass}, and Lajos Takács \cite{Takacs} in 1961--62. See also Theorem~2 in the survey \cite{AddarioBerryReed}.

\section{Spanning trees in bipartite graphs} \label{S:bipartite}

The following enumeration of spanning trees in a complete bipartite graph was established in \cite{Austin, Scoins}, see also \cite{HartsfieldWerth, Jaworsky, KleeStamps, Lewis, Moon}.

\begin{Prop} \label{P:bipartite}
The number of spanning trees in the complete bipartite graph $K_{m, n}$ is \[m^{n-1}\cdot n^{m-1}.\] 
\end{Prop}

We can prove Proposition~\ref{P:bipartite} through a bipartite version of padlock solitaire. Suppose there are two rows of boxes labeled $A_1,\dots, A_m$ and $B_1,\dots, B_n$. Each key from row $A$ is thrown into a randomly chosen box of row $B$ and vice versa, except that we keep the key to box $A_1$ to start the game.

Again the winning positions can be represented as trees rooted at box $A_1$ and spanning the set of all boxes, but now only the trees that respect the bipartition into the two rows are counted. Since there are $m^n\cdot n^{m-1}$ ways of distributing the $m-1$ remaining keys in row $A$ into row $B$, and the $n$ keys from row $B$ into row $A$, Proposition~\ref{P:bipartite} is amounts to showing the following:

\begin{Thm} \label{T:simpleBipartite}
Suppose that the keys to $A_2,\dots, A_m$ and $B_1,\dots,B_n$ are placed independently, each according to uniform distribution on the boxes of the opposite row. Then the probability that we can unlock everything starting from key $A_1$ is $1/m$.
\end{Thm}

\begin{proof}
The game can be analyzed in ``rounds'', where in each round we use all available keys. This means that we alternate between holding only keys to row $A$ and holding only keys to row $B$. It turns out, by now not surprisingly, that after each round the probability of winning is  
\begin{equation} \label{bipartite} \frac{\text{\# available keys (all from one row)}}{\text{\# unopened boxes (in that row)}}.\end{equation}
This is again because the expected change in \eqref{bipartite} in a round is zero: For each key in the other row, \eqref{bipartite} is precisely the probability of that key being in one of the boxes we're about to open.

Consequently the probability of winning if we start with one out of $m$ keys in row $A$ is $1/m$. 
\end{proof}

There is an amusing puzzle of geometric rigidity described in \cite{BolkerCrapo} and called \emph{Bracing the Grid} in \cite{Puzzles}, that involves spanning trees in bipartite graphs. 

\section{Spanning trees in multi-partite graphs} \label{S:multipartite}
Proposition~\ref{P:bipartite} can be generalized to counting trees in multi-partite graphs. Suppose a set of $n$ labeled vertices are partitioned into $k$ parts of sizes $n_1,\dots, n_k$, where $n = n_1+ \dots +n_k$. The complete $k$-partite graph, denoted $K(n_1,\dots, n_k)$, has an edge between every pair of vertices from distinct parts. This means that its spanning trees are the trees that connect all vertices without any edge between two vertices from the same part. 

The following generalization of Proposition~\ref{P:bipartite} (and of the Cayley formula!) was proved in \cite{Austin} and \cite{Onodera}, see also \cite{KleeStamps, Lewis}.
\begin{Prop} \label{P:multipartite}
The number of spanning trees in the complete $k$-partite graph $K(n_1,\dots, n_k)$ is 
\[ n^{k-2}\cdot(n-n_1)^{n_1-1}\ \cdots\ (n-n_k)^{n_k-1}.\]
\end{Prop}

To obtain a proof using padlock solitaire, we arrange $n$ boxes in $k$ rows with $n_1,\dots, n_k$ boxes in each row respectively. We retain the key to the first box of the first row, and distribute the remaining keys independently, each key uniformly distributed between the boxes \emph{of the other rows}. 

The number of ways of distributing the keys is 
\[  
(n-n_1)^{n_1-1}\cdot (n-n_2)^{n_2}\ \cdots\ (n-n_k)^{n_k},
\] 
and again the winning starting positions correspond exactly to the trees we wish to count. In order to establish Proposition~\ref{P:multipartite}, we therefore want to show that our winning probability is

\begin{equation} \label{eq:multipartite}
\frac{n^{k-2}} { (n-n_2)(n-n_3)\ \cdots\ (n-n_k)} = \frac1n\cdot \frac{n}{n-n_2}\ \cdots \ \frac{n}{n-n_k}. 
\end{equation}

We present a proof based on what might first seem like pulling the expression \eqref{winning} below out of a hat, and only then argue that the formula is quite natural in view of our earlier results.

A \emph{state} of the game of $k$-row padlock solitaire is given by the numbers $H_i$ of \emph{hidden} keys of row $i$, and $L_i$ of locked boxes in row $i$ (counting also those to which we already have the key), for $i=1,\dots, k$.

We define $X_i$ (implicitly depending on the state) as the average number of keys to row $i$ in the boxes of the other rows. Letting $L = L_1+\dots + L_k$, since there are $H_i$ hidden keys to row $i$ and $L-L_i$ boxes where they can be, 
\[ X_i = \frac{H_i}{L-L_i}.\] 
Notice that we can stop the game before any denominator becomes zero: If we have the option of opening the last box of a row when only two rows remain, it's already clear that we are winning. 

\begin{Lemma} \label{L:multipartite}
From an arbitrary state, the winning probability is given by
\begin{equation} \label{winning} f_k(X_1, \dots, X_k) = (1+X_1)\ \cdots \ (1+X_k)\cdot\left[ 1 - \frac{X_1}{1+X_1} - \dots - \frac{X_k}{1+X_k}\right]. \end{equation}
\end{Lemma}

\begin{proof}
By multiplying out in \eqref{winning}, we see that there aren't really any denominators, and that $f_k$ is a polynomial where each term is square-free. This implies that $f_k$ is a martingale under the operation of opening a box: The hidden keys are equally likely to be in all boxes of the other rows, and therefore the expected change in value of $X_i$ as we open a box is zero. Since the keys to different rows are distributed independently, every product of distinct $X_i$'s has zero expected change too.

Moreover, $f_k$ becomes 0 whenever we get stuck: If $H_i = L_i$ for every $i$, then \[ \frac{X_i}{1+X_i} = \frac{L_i}{L},\] and the rightmost factor  of \eqref{winning} is zero.   

Finally, $f_k$ becomes 1 if we win: If we recover the keys to all rows except one, then all except one of $X_1,\dots,X_k$ become zero. If for instance only $X_1$ remains, then \eqref{winning} becomes   
\[ (1+X_1) \cdot \left[1-\frac{X_1}{1+X_1}\right] = 1.\]

These properties together imply that $f_k$ gives the winning probability from every state. 
\end{proof}

\begin{proof} [Proof of Proposition~\ref{P:multipartite}]
Now we can deduce \eqref{eq:multipartite} by plugging in the starting position where we have one key to row 1 but no other keys. For this state we have $L_i = n_i$, $H_1 = n_1-1$, and $H_i = n_i$ for $i\geq 2$. With these values, \[1+X_i = \frac{n}{n-n_i},\] except that \[1+X_1 = \frac{n-1}{n-n_1}.\]
Therefore \eqref{winning} becomes
\begin{equation} 
\frac{n-1}{n-n_1}\cdot \frac{n}{n-n_2} \ \cdots \ \frac{n}{n-n_k}\cdot\left[1 - \frac{n_1-1}{n-1} - \frac{n_2}{n} - \cdots - \frac{n_k}{n}\right],
\end{equation}
which simplifies to \eqref{eq:multipartite}. 
\end{proof}

Let us briefly comment on how one might arrive at the formula \eqref{winning}. In view of Theorem~\ref{T:dependent} we might suspect that the winning probability, also in the multi-row version, remains the same even if to some extent the keys are not distributed independently. In particular, we might conjecture that the winning probability remains the same under a \emph{key-ring assumption}: conditioning on keys from the same row always ending up in the same box. We can think of this as grouping the keys on key-rings, one for each row, and then distributing the key-rings independently, each to a box chosen uniformly among the unopened boxes of the other rows.  

In retrospect we can see that indeed \eqref{winning} gives the probability of winning also under the key-ring assumption. All we need for the proof of Lemma~\ref{L:multipartite} to work is that each hidden key is equally likely to be in any of the unopened boxes of the other rows, and that every set of keys \emph{from distinct rows} are distributed independently of each other. Under any such scheme, square-free monomials in $X_1,\dots, X_k$ are martingales. 

But if the keys are arranged on key-rings, we win precisely when no set of them are locked cyclically into each other's rows (in boxes to which we don't already have the key). 

Looking at the case that, say, key-ring 1 is locked into row 2, key-ring 2 is locked into row 3, and key-ring 3 into row 1, we see that the probability that this happens is 
\[ \frac{H_2}{L - L_1} \cdot \frac{H_3}{L - L_2}\cdot \frac{H_1}{L - L_3},\] since for instance there are $H_2$ boxes in row 2 to which we don't have the key, and key-ring 1 can be placed anywhere except in row 1. Cyclically shifting the numerators, we can write this as 
\[\left( \frac{H_1}{L - L_1}\right) \cdot \left(\frac{H_2}{L - L_2}\right) \cdot \left(\frac{H_3}{L - L_3}\right) = X_1X_2X_3.\]
Similarly every probability of cyclically locking in a set of key-rings can be expressed as a square-free monomial in $X_1,\dots, X_k$, and by inclusion-exclusion there must be a formula for the probability that there is no such cycle. 

For $k=2$ and $k=3$ these formulas are $1-X_1X_2$ and $1-X_1X_2 - X_1X_3 - X_2X_3 - 2X_1X_2X_3$ respectively, from which we can guess that the general form is \eqref{winning}. 

The formula \eqref{winning} can also be written as the determinant
\[ \det 
\begin{pmatrix}
1 - X_{1} & -X_{1} & \cdots & -X_{1}\\
-X_{2} & 1- X_{2} & \cdots & -X_{2}\\
\vdots & \vdots & \ddots & \vdots \\
-X_{k} & -X_{k} & \cdots & 1-X_{k}
\end{pmatrix}
\] 
but now we have come almost full circle and are about to rediscover the matrix tree theorem!

\section{Nilpotent matrices over finite fields}

A square matrix is \emph{nilpotent} if some power of it is zero. The following was proved by Nathan Fine and Israel Herstein \cite{FineHerstein} in 1958: 
\begin{Prop} \label{P:nilpotent}
The number of nilpotent $n\times n$-matrices over the field $F_q$ of $q$ elements is
\[ q^{n(n-1)}.\]
\end{Prop}
According to Brouwer, Gow and Sheekey \cite{BrouwerGowSheekey}, a proof was given in lecture notes by Philip Hall already in 1955. 

It was pointed out in \cite{Crabb} that Proposition~\ref{P:nilpotent} can be regarded as a generalization of the Cayley formula, where trees correspond to nilpotent mappings of vector spaces over the mythical ``field of one element'' $F_1$. 

We give a virtually calculation-free proof by setting up a game of padlock solitaire with $q^n$ boxes, one for each element of the vector space $V = F_q^n$. We distribute the keys by choosing uniformly a random $n$ by $n$ matrix $A$ and putting the key to box $v$ into the box $f(v) = A\cdot v$, except that we keep the key to the zero box.

Notice that we can unlock all boxes if and only if the matrix $A$ is nilpotent so that iterating the function $f$ eventually leads to mapping everything to the zero box. Since there are $q^{n^2}$ matrices of dimension $n$ by $n$, establishing Proposition~\ref{P:nilpotent} amounts to showing that the winning probability is $1/q^n$. 
This is the reciprocal of the number of boxes, and therefore precisely what we would expect in view of Theorems~\ref{T:simple} and \ref{T:dependent}, but distributing the keys according to a linear function doesn't quite satisfy the conditions of these theorems.  

To establish a martingale property of the keys-to-boxes ratio, we open the boxes in rounds as in Section~\ref{S:bipartite} (and this time it's actually necessary) where we simultaneously open all the boxes to which we have the key. Identifying keys and boxes with elements of $V$, at every stage the set $K$ of recovered keys is a linear space, and the set $U$ of already opened boxes is a linear subspace, except at the very beginning when $U$ is empty. 

For the very first step, it's clear that each nonzero key has probability $1/q^n$ of being in the zero box, since as soon as a vector $v$ has a nonzero coordinate, the product $A\cdot v$ will be uniformly distributed over $V$. 

At a generic stage where $U$ is nonempty, we have $U\subseteq K$ and $K = f^{-1}(U)$. Suppose that $k = \dim(K)$ and that $v_1,\dots, v_k$ is a basis for $K$. Having opened the boxes of $U$, we know all values $f(v)$ for $v\in K$, and they are determined by $f(v_1),\dots, f(v_k)$. Moreover we know that no vector outside $K$ is mapped to $U$. 

If we extend the basis $v_1,\dots, v_k$ for $K$ to a basis $v_1,\dots, v_n$ for the whole space $V$, then the functions $f$ which are compatible with what we have seen in the boxes of $U$ are those where $f(v_{k+1}) \notin U$, $f(v_{k+2})$ is not in the span of $U$ and $f(v_{k+1})$, and so on, generally $f(v_{k+i})$ not belonging to the span of $U$ and $f(v_{k+1}),\dots, f(v_{k+i-1})$.

We could write down an expression for the number of such functions $f$ in terms of the dimensions of $K$ and $U$, but all we need is that the number of options for $f(v_{k+i})$ once that $f(v_{k+1}),\dots, f(v_{k+i-1})$ have been chosen is the same regardless of those earlier choices. 

This means that regardless of how we choose $f(v_{k+1})$ among the elements of $V-U$, the number of functions satisfying what we already know will be the same. In particular, given what we know, $f(v_{k+1})$ is just as likely to be a given element of $K$ as a given element not in $K$.    

Since we can choose $v_{k+1}$ to be any element we like outside $K$, this shows that after a complete round, all hidden keys are distributed uniformly between the unopened boxes. This establishes the martingale property of the average number of keys in the unopened boxes, and thereby Proposition~\ref{P:nilpotent}.

\section{Nilpotent matrix products}
Finally we prove a result that we haven't found in the literature.
The following is a bipartite version of Proposition~\ref{P:nilpotent} in the same way that Proposition~\ref{P:bipartite} is a bipartite version of the Cayley formula.

\begin{Prop} \label{P:bipartiteNilpotent}
Let $A$ and $B$ be random matrices of dimensions $m$ by $n$ and $n$ by $m$ respectively, over the same finite field $F_q$ and chosen independently and uniformly over all such matrices. Then the probability that the product $AB$ is nilpotent is 
\begin{equation} \label{eq:bipartiteNilpotent}
\frac1{q^m} + \frac1{q^n} -\frac1{q^{m+n}}.
\end{equation}
\end{Prop}

\begin{proof}
Let $V_1$ and $V_2$ be vector spaces over $F_q$ of dimensions $m$ and $n$ respectively. The matrix product $AB$ is nilpotent if and only if iterating the corresponding linear mappings back and forth between $V_1$ and $V_2$ eventually leads to the zero function. It doesn't matter where we start, and in particular $AB$ is nilpotent if and only if $BA$ is.   

We set up a game of padlock solitaire with one box for every element of $V_1$ and one for every element of $V_2$. We keep the keys to the two zero boxes and distribute the remaining keys from each vector space into the boxes of the other one by the random linear functions given by $A$ and $B$. We recover all the keys if and only if the composition of $A$ and $B$  (in any order) is nilpotent.  

Again we open the boxes in rounds, and for the same reason as in the proof of Proposition~\ref{P:nilpotent}, after opening a linear subspace of boxes in each of the two spaces, each hidden key will be distributed uniformly between all unopened boxes of the other space. 

We let $X_1$ and $X_2$ be the average number of keys from one space in the unopened boxes of the other one respectively. Opening the boxes in rounds, it follows as in Section~\ref{S:multipartite} that the quantity 
\[ 1 - X_1X_2\] 
is a martingale which ends at 1 if we win and 0 if we lose. 

Since we start by distributing the $q^m-1$ nonzero keys from $V_1$ into the $q^n$ boxes of $V_2$, and the $q^n-1$ nonzero keys from $V_2$ into the $q^m$ boxes of $V_1$, the initial winning probability is 
\[
1 - \frac{q^m-1}{q^n}\cdot \frac{q^n-1}{q^m},\]
which simplifies to \eqref{eq:bipartiteNilpotent}.
\end{proof}

\end{document}